\documentclass[10pt]{elsarticle}

\newtheorem{proposition}{Proposition}

\usepackage{latexsym}
\usepackage{amsmath}
\usepackage{times}
\usepackage{graphicx,graphics, epsfig}
\usepackage{amsfonts}
\DeclareGraphicsExtensions{.pdf,.png,.jpg}

\newenvironment{proof}[1][Proof]{\textbf{#1.} }{\ \rule{0.5em}{0.5em}}

\begin{document}

\begin{frontmatter}

\title{Dynamics of a family of Chebyshev-Halley type methods \tnoteref{thanks}}
\tnotetext[thanks]{This research was supported by Ministerio de
Ciencia y Tecnolog\'{i}a MTM2011-28636-C02-02 and by Vicerrectorado
de Investigaci\'on. Universitat Polit\`ecnica de Val\`encia
PAID-06-2010-2285}

\author[IMM]{Alicia Cordero}
\ead{acordero@mat.upv.es}

\author[IMM]{Juan R. Torregrosa \corref{cor1}}
\ead{jrtorre@mat.upv.es}

\author[IMAC]{Pura Vindel}
\ead{jlhueso@mat.upv.es}

\cortext[cor1]{Corresponding author}

\address[IMM]{Instituto Universitario de Matem\'atica Multidisciplinar \\
Universitat Polit\`ecnica de Val\`encia \\ Camino de Vera s/n, 46022 Val\`encia, Spain  }

\address[IMAC]{Instituto de Matem\'aticas y Aplicaciones de Castell\'on \\
Universitat Jaume I \\
Campus de Riu Sec s/n, Castell\'on, Spain }

\begin{abstract}
In this paper, the dynamics of the Chebyshev-Halley family is studied on quadratic polynomials. A singular set, that we call cat set,
appears in the parameter space associated to the family. This cat set has interesting similarities with the Mandelbrot set.
The parameters space has allowed us to find different elements of the family such that can not converge to any root
of the polynomial, since periodic orbits and attractive
strange fixed points appear in the dynamical plane of the corresponding method.


\end{abstract}
\end{frontmatter}


\section{Introduction}
The application of iterative methods for solving nonlinear equations
$f(z)=0$, with  $f:\mathbb{C}\rightarrow \mathbb{C}$, give rise
to rational functions whose dynamics are not well-known. The
simplest model is obtained when $f(z)$ is a quadratic polynomial and
the iterative process is Newton's method. The study on the dynamics
of Newton's method has been extended to other point-to-point
iterative methods used for solving nonlinear equations, with
convergence order up to three (see, for example \cite{ABBP},
\cite{ABP} and, more recently, \cite{GHR} and \cite{PR}).

The most of the well-known point-to-point cubically convergent
methods belong to the one-parameter family, called Chebyshev-Halley
family,
\begin{equation}\label{ChH}
z_{n+1}=z_{n}-\left( 1+\frac{1}{2}\frac{L_{f}\left( z_{n}\right)
}{1-\alpha L_{f}\left( z_{n}\right) }\right) \frac{f\left(
z_{n}\right) }{f^{\prime }\left( z_{n}\right) },
\end{equation}
where
\begin{equation}
L_{f}\left( z\right) =\frac{f\left( z\right) f^{\prime \prime
}\left( z\right) }{\left( f^{\prime }\left( z\right) \right) ^{2}},
\end{equation}
and $\alpha$ is a complex parameter. This family includes
Chebyshev's method for $\alpha=0$, Halley's scheme for
$\alpha=\frac{1}{2}$, super-Halley's method for $\alpha=1$ and
Newton's method when $\alpha$ tends to $\pm\infty$. As far as we know,
this family was already studied by Werner in 1981 (see \cite{werner}),
and can also be found in \cite{argyros} and \cite{TR}. Moreover, a
geometrical construction in studied in \cite{ABG}. It is
interesting to note that any iterative process given by the
expression:
\begin{equation}
z_{n+1}=z_{n}-H\left( L_{f}\left( z_{n}\right)\right) ,
\end{equation}
where function $H$ satisfies $H(0)=0$, $H'(0)=\frac{1}{2}$ and
$|H''(0)|<\infty$, generates an order three iterative method (see
\cite{gander}).

The family of Chebyshev-Halley has been widely analyzed under
different points of view. For example, in \cite{HS1}, \cite{HS2} and
\cite{YLS}, the authors studied the conditions under the global and
semilocal convergence of this family in Banach spaces are hold.
Also, the semilocal convergence of this family in the complex plane
is presented in \cite{GH}.

Many authors have introduced different variants of these family, in
order to increase its applicability and its order of convergence.
For instance, Osada in \cite{osada} showed a variant able to find
the multiple roots of analytic functions and a procedure to obtain
simultaneously all the roots of a polynomial. On the other hand, in
\cite{K}, \cite{KL} and \cite{AHR} the authors got multipoint
variants of the mentioned family with sixth order of convergence.
Another trend of research about this family have been to avoid the
use of second derivatives (see \cite{X} and \cite{C}) or to design
secant-type variants (see \cite{AEGHH}).

From the numerical point of view, the dynamical behavior of the
rational function associated to an iterative method give us
important information about its stability and reliability. In this
terms, Varona in \cite{V} described the dynamical behavior of
several well-known iterative methods. More recently, in \cite{GHR}
and \cite{HPR}, the authors study the dynamics of different
iterative families.

\subsection{Basic concepts}

The fixed point operator corresponding to the family of Chebyshev-Halley described in (\ref{ChH}) is:
\begin{equation} \label{g}
G\left( z\right) =z-\left( 1+\frac{1}{2}\frac{L_{f}\left( z\right) }{%
1-\alpha L_{f}\left( z\right) }\right) \frac{f\left( z\right) }{f^{\prime
}\left( z\right) }.
\end{equation}
In this work, we study the dynamics of this operator when it is
applied on quadratic polynomials. It is known that the roots of a
polynomial can be transformed by an affine map with no qualitative
changes on the dynamics of family (\ref{ChH}) (see \cite{douady}). So, we can use the
quadratic polynomial $p\left( z\right) = z^{2}+c$. For $p(z)$, the
operator (\ref{g}) corresponds to the rational function:

\begin{equation}
G_{p}\left( z\right) =\frac{z^{4}\left( -3+2\alpha \right)
+6cz^{2}+c^{2}\left( 1-2\alpha \right) }{4z\left( z^{2}\left(
-2+\alpha \right) +\alpha c\right) },
\end{equation}
depending on the parameters $\alpha $ and $c$.

P. Blanchard, in \cite{blanchard}, by considering the conjugacy  map
\begin{equation}
h\left( z\right) =\frac{z-i\sqrt{c}}{z+i\sqrt{c}},
\label{conjugacy}
\end{equation}
with the following properties:
\begin{equation*}
\mbox{i)} \ \  h\left( \infty \right) =1, \ \  \ \ \mbox{ii)} \ \  h\left( i\sqrt{c}\right) =0, \ \  \ \ \mbox{iii)} \ \  h\left( -i\sqrt{c}\right) =\infty ,
\end{equation*}
proved that, for quadratic polynomials, the Newton's operator is
always conjugate to the rational map $z^{2}$. In an analogous way,
it is easy to prove, by using the same conjugacy map, that the
operator $G_{p}\left( z\right)$ is conjugated to the operator
$O_p(z)$
\begin{equation}
O_{p}\left( z\right) =\left( h\circ G_{p}\circ h^{-1}\right) \left(
z\right) =z^{3}\frac{z-2\left( \alpha -1\right) }{1-2\left( \alpha
-1\right) z}. \label{op}
\end{equation}
In addition, the parameter $c$ has been obviated in $O_p(z)$.

In this work, we study the general convergence of methods
(\ref{ChH}) for quadratic polynomials. To be more precise (see
\cite{SM} and \cite{MC}), a given method is generally convergent if
the scheme converges to a root for almost every starting point and
for almost every polynomial of a given degree.

\subsubsection{Dynamical concepts}

Now, let us recall some basic concepts on complex dynamics (see \cite{blanchard2}). Given a
rational function $R:\hat{\mathbb{C}}\rightarrow \hat{\mathbb{C}}$,
where $\hat{\mathbb{C}}$ is the Riemann sphere, the \emph{orbit of a
point} $z_{0} \in \hat{\mathbb{C}}$ is defined as:
\[
\quad z_{0},\,R\left( z_{0}\right) ,\,R^{2}\left( z_{0}\right)
,...,R^{n}\left( z_{0}\right) ,...
\]
We are interested in the study of the asymptotic behavior of the
orbits depending on the initial condition $z_{0},$ that is, we are
going to analyze the phase plane of the map $R$ defined by the
different iterative methods.

To obtain these phase spaces, the first of all is to classify the
starting points from the asymptotic behavior of their orbits.

A $z_{0} \in \hat{\mathbb{C}}$ is called a \emph{fixed point} if it
satisfies: $R\left( z_{0}\right)=z_{0}$. A \emph{periodic point}
$z_{0}$ of period $p>1$ is a point such that $R^{p}\left(
z_{0}\right) =z_{0}$ and $R^{k}\left( z_{0}\right) \neq z_{0}$,
$k<p$. A \emph{pre-periodic point} is a point $z_{0}$ that is not
periodic but there exists a $k>0$ such that $R^{k}\left(
z_{0}\right) $ is periodic. A \emph{critical point} $z_{0}$ is a
point where the derivative of rational function  vanishes,
$R^{\prime }\left( z_{0}\right) =0$.

On the other hand, a fixed point $z_{0}$ is called \emph{attractor}
if $|R^{\prime
}(z_{0})|<1$,  \emph{superattractor} if $|R^{\prime }(z_{0})|=0$, \emph{repulsor} if $%
|R^{\prime }(z_{0})|>1$ and \emph{parabolic} if $|R^{\prime
}(z_{0})|=1$.

\emph{The basin of attraction} of an attractor $\alpha$ is defined as the set of pre-images of any order:
\[
\mathcal{A}\left( \alpha \right) =\{z_{0}\in \hat{\mathbb{C}} \ : \
R^{n}\left( z_{0}\right) {\rightarrow }\alpha, \ n{\rightarrow
}\infty \}.
\]

The set of points $z\in \hat{\mathbb{C}}$ such that their families
$\left\{ R^{n}\left( z\right) \right\} _{n\in \Bbb{N}}$ are normal
in some neighborhood $U\left( z\right) ,$ is the \emph{Fatou set,} $\mathcal{F}%
\left( R\right) ,$ that is, the Fatou set is composed by  the set of
points whose orbits tend to an attractor (fixed point, periodic
orbit or infinity). Its complement in $\hat{\mathbb{C}}$ is the
\emph{Julia set,} $\mathcal{J}\left( R\right) ;$ therefore,  the
Julia set includes all repelling fixed points, periodic orbits and
their pre-images. That means that the basin of attraction of any
fixed point belongs to the Fatou set. On the contrary, the
boundaries of the basins of attraction belong to the Julia set.

The invariant Julia set for Newton's method is the unit circle
$S^1$ and the Fatou set is defined by the two basins of atraction
of the superattractor fixed points: $0$ and $\infty$. On the other
hand, the Julia set for Chebyshev's method applied to quadratic
polynomials is more complicated than for Newton's method and it has
been studied in \cite{kn}. These methods are two elements of the
family (\ref{ChH}). In the following sections, we look for the Julia
and Fatou sets for the rest of the elements of the mentioned family.

The rest of the paper is organized as follows: in Section 2  and 3 we
study the fixed and critical points, respectively, of the operator $O_p(z)$. The dynamical behavior of the family (\ref{ChH}) is analyzed
in Section 4. We finish the work with some remarks and conclusions.

\section{Study of the fixed points}

We are going to study the dynamics of the operator $O_{p}\left(
z\right)$ in function of the parameter $\alpha$. In this section, we
calculate the fixed points of $O_{p}\left( z\right)$ and in the next
one, its critical points. As we will see, the number and the
stability of the fixed and critical points depend on the parameter
$\alpha$.

The fixed points of $O_{p}\left( z\right)$ are the roots of the
equation $O_{p}\left( z\right)=z$, that is, $z=0$, $z=1$ and
\begin{equation}
z=\frac{-3+2\alpha \pm \sqrt{ 5-12\alpha +4\alpha ^{2}}}{2},
\end{equation}
which are the two roots of $z^2+(3-2\alpha)z+1=0$ denoted by
$s_1$ and $s_2$.

The number of the finite fixed points depends on $\alpha $.
Moreover, $s_{1}=\dfrac{1}{s_{2}}$, so that, these points are
equal only if $s_{1}=s_{2}=\pm 1$; this happens when
$5-12\alpha+4\alpha^2=0$, i.e., for $\alpha=\frac{1}{2}$ and
$\alpha=\frac{5}{2}$.

For $\alpha =\frac{1}{2}$, $s_1=s_2=-1$, so $z=-1$ is a fixed point
with double multiplicity. For $\alpha =\frac{5}{2}$, $s_1=s_2=1$, so
$z=1,$ has multiplicity 3.

Summarizing:
\begin{itemize}
\item If $\alpha \neq \frac{1}{2}$ and $\alpha \neq\frac{5}{2}$, there are five different fixed points with
multiplicity 1.
\item If $\alpha =\frac{1}{2},$ there are four different
fixed points: $z=0$, $z=\infty$ and $z=1$ with multiplicity 1 and
$z=-1$ with multiplicity 2.
\item If $\alpha=\frac{5}{2},$ there are 3 different fixed points: $z=0$ and $z=\infty$ with multiplicity 1 and
$z=1$ with multiplicity 3.
\end{itemize}
As we will see in the next section, the multiplicity of the fixed points implies different
dynamical behaviors.

In order to study the stability of the fixed points, we calculate the first derivative of $O_p(z)$,

\begin{equation}\label{opder}
O_{p}^{\prime }\left( z\right) =2z^{2}\frac{3\left( 1-\alpha \right)
+2z\left( 3-4\alpha +2\alpha ^{2}\right) +3z^{2}\left( 1-\alpha \right) }{
\left( 1-2\left( \alpha -1\right) z\right) ^{2}}.
\end{equation}

From $\left( \ref{opder}\right) $ we obtain that the origin and $\infty $
are always superattractive fixed points, but the stability of the other
fixed points changes depending on the values of the parameter $\alpha$.
These points are called \emph{strange fixed points}.

The operator $O_{p}^{\prime }\left( z\right)$ in $z=1$ gives
\begin{equation}\label{stabilidaddel1}
\left| O_{p}^{\prime }\left( 1\right) \right| =\left| 4\frac{(-2+\alpha)(2\alpha -3) }{(2\alpha -3)^2}\right|=\left| \frac{4\alpha -8}{2\alpha -3}\right| .
\end{equation}

If we analyze this function, we obtain an horizontal asymptote in $\left| O_{p}^{\prime }\left( 1\right) \right|= 2$, when $\alpha
\rightarrow \pm \infty$, and a vertical asymptote in $\alpha=\frac{3}{2}$.
In the following result we present the stability of the fixed point $z=1$.

\begin{proposition}\label{lemaestabilidad1}
The fixed point $z=1$ satisfies the
following statements :
\begin{itemize}
\item[i)]  If $\left| \alpha -\frac{13}{6}\right| <\frac{1}{3}$, then $z=1$ is an
attractor and, in particular, it is a superattractor for $\alpha =2$.
\item[ii)]  If $\left| \alpha -\frac{13}{6}\right| =\frac{1}{3}$, then $z=1$ is a parabolic point.
\item[iii)] If $\left| \alpha -\frac{13}{6}\right| > \frac{1}{3}$, then $z=1$ is a repulsive fixed point.
\end{itemize}
\end{proposition}
\proof From equation (\ref{stabilidaddel1}),
\[
\left| 4\frac{-2+\alpha }{2\alpha
-3}\right| \leq 1\Rightarrow 4\left| -2+\alpha\right| \leq \left| 2\alpha
-3\right| .
\]

Let $\alpha =a+ib$ be an arbitrary complex number. Then,
\[
\left| -2+\alpha\right| ^{2}=\left(
-2+a\right) ^{2}+b^{2}
\]
and
\[
\left| 2\alpha -3\right| ^{2}=\left(
2a-3\right) ^{2}+4b^{2}.
\]
So,
\[16\left( 4-4a+a^{2}+b^{2}\right) \leq
4a^{2}-12a+9+4b^{2}.
\]
By simplifying
\[
55-52a+12a^{2}+12b^{2}= 12\left( a-\dfrac{13}{6}\right) ^{2}+12b^{2}-\frac{4}{3}\leq 0,
\]
that is,
\[
  \left( a-\dfrac{13}{6}\right) ^{2}+b^{2}\leq \dfrac{1}{9}.
\]
Therefore,
\[
\left| O_{p}^{\prime }\left( 1\right) \right| \leq 1 \ \ \mbox{if  and only if} \ \ \left| \alpha -\frac{13}{6}\right| \leq \frac{1}{3}.
\]
Finally, if $\alpha$ satisfies $\left| \alpha -\frac{13}{6}\right| > \frac{1}{3}$, then $\left| O_{p}^{\prime }\left( 1\right) \right| >1$
and $z=1$ is a repulsive point.
\endproof

The stability of the other strange fixed points $z=s_i$, $i=1,2$ also depends on parameter $\alpha$.
\begin{equation*}\label{estExtr}
\left| O_{p}^{\prime }( s_i) \right| =\left| 6-2\alpha \right|.
\end{equation*}
We can establish the following result:

\begin{proposition} \label{lemaestabilidadext}
The fixed points $z=s_i$, $i=1,2$ satisfy the following statements:
\begin{itemize}
\item[i)]  If $\left| \alpha -3\right| <\frac{1}{2}$, then $s_{1}$ and $s_{2}$
are two different attractive fixed points. In particular, for $\alpha =3$,
$s_{1} $ and $s_{2}$ are superattractors.
\item[ii)]  If $\left| \alpha -3\right| =\frac{1}{2}$, then $s_{1}$ and $s_{2}$ are parabolic points. In particular, for $\alpha =\frac{5}{2}$, $
s_{1}=s_{2}=1.$
\item[iii)] If $\left| \alpha -3\right| > \frac{1}{2}$, then $s_{1}$ and $s_{2}$ are repulsive
fixed points.
\end{itemize}
\end{proposition}

In the following bifurcation diagram (Figure \ref{bifpuntosfijos}) we represent the behavior of the fixed point for real values of parameter $\alpha$.
The point $z=\infty$ is not represented. Let us observe that the stability of the fixed points is represented by the thickness of the lines: if it is attractive,
the line corresponding to the value of this strange point is thicker. So, it can be noticed that $z=0$ is always an attractor, meanwhile $z=1$ is
attractive when $\frac{11}{6}<\alpha<\frac{5}{2}$ and $s_i$, $i=1,2$ are
attractors only when $\frac{5}{2}<\alpha<\frac{7}{2}$.

\begin{figure}[h]
\begin{center}
  \includegraphics[width=8cm]{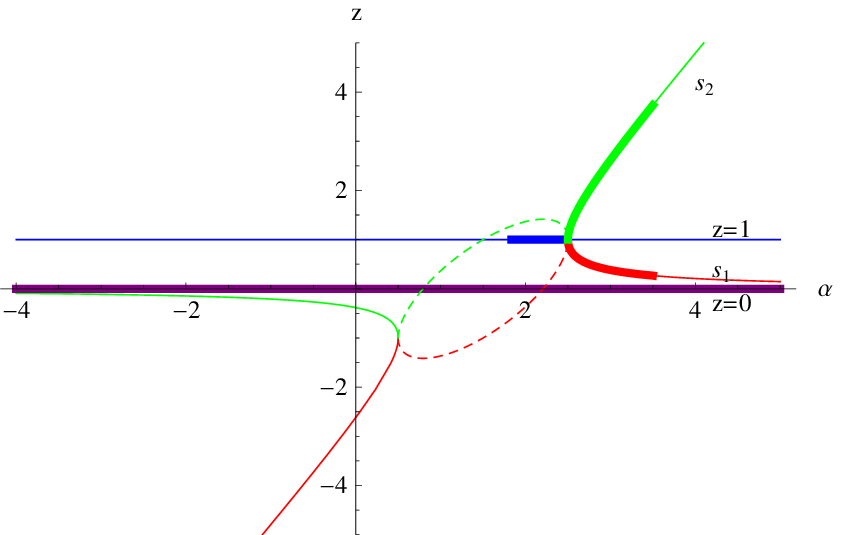}
  \caption{Bifurcation diagram of fixed
points}\label{bifpuntosfijos}
\end{center}
\end{figure}

\section{Study of the critical points}

Let us remember that the critical points of $O_p(z)$ are the roots of $O_p^{\prime }(z)=0$, that is, $z=0$, $z=\infty$, and
\begin{equation}\label{criticos}
z=\frac{3-4\alpha +2\alpha
^{2}\pm \sqrt{-6\alpha +19\alpha ^{2}-16\alpha ^{3}+4\alpha ^{4}}}{3\left(
\alpha -1\right) },
\end{equation}
which are denoted by $c_{1}$ and $c_{2}$.

It is easy to prove that $c_{1}=\dfrac{1}{c_{2}}$. Therefore, both critical points
coincide only when $c_{1}=c_{2}=\pm 1$, that is, when

\begin{equation*}
   -6\alpha +19\alpha^{2}-16\alpha ^{3}+4\alpha^{4}=0.
\end{equation*}
The roots of this equation are $0$, $\dfrac{1}{2}$, $\dfrac{3}{2}$ and $2$.

It is known that there is at least one critical point associated with each invariant Fatou component. As $z=0$ and $z=\infty$ are
both superattractive fixed points of $O_p(z)$, they also are critical points and give rise to their respective Fatou components.
For the other critical points, we can establish the following remarks:

\begin{itemize}
\item[a)]  If $\alpha =0$, then $c_{1}=c_{2}=-1$, and it is a pre-image of the fixed point $z=1$: $O_{p}\left( -1\right) =1$. As $z=1$ is repulsive,
$z=-1 \in \mathcal{J}(O_p)$. So, $O_p(z)$ has precisely two invariant Fatou components, $\mathcal{A}\left( 0 \right)$
and $\mathcal{A}\left( \infty\right)$.
\item[b)]  If $\alpha =\frac{1}{2}$, then $c_{1}=c_{2}=-1=s_{1}=s_{2}$ are
repulsive fixed points and belong to Julia set.
\item[c)]  If $\alpha =\frac{3}{2}$, then $c_{1}=c_{2}=1$ is a repulsive fixed point and belong to Julia set.
\item[d)]  If $\alpha =2$, $c_{1}=c_{2}=1$. In this case $z=1$ is a superattractor, which gives rise to a Fatou component.
\item[e)]  For any other value of $\alpha \in \mathbf{C}$, there are four different critical points, and we will study their behavior in Section 4.
\end{itemize}

In Figure \ref{ext}, we represent the behavior of the strange fixed points
and critical points for real values of $\alpha$ between $1$ and $4$. We observe that the critical points $c_i$, $i=1,2$ are inside the
basin of attraction of $z=1$ when it is attractive ($\frac{11}{6}<\alpha<\frac{5}{2}$) and coincide with $z=1$ for $\alpha=2$. Then,
they move to the basins of attraction of $s_1$ and $s_2$ when these fixed points become attractive ($\frac{5}{2}<\alpha<\frac{7}{2}$),
critical and fixed points coincide for $\alpha=3$ and $s_1$ and $s_2$  become superattractors.
\begin{figure}[h]
\begin{center}
  \includegraphics[width=8cm]{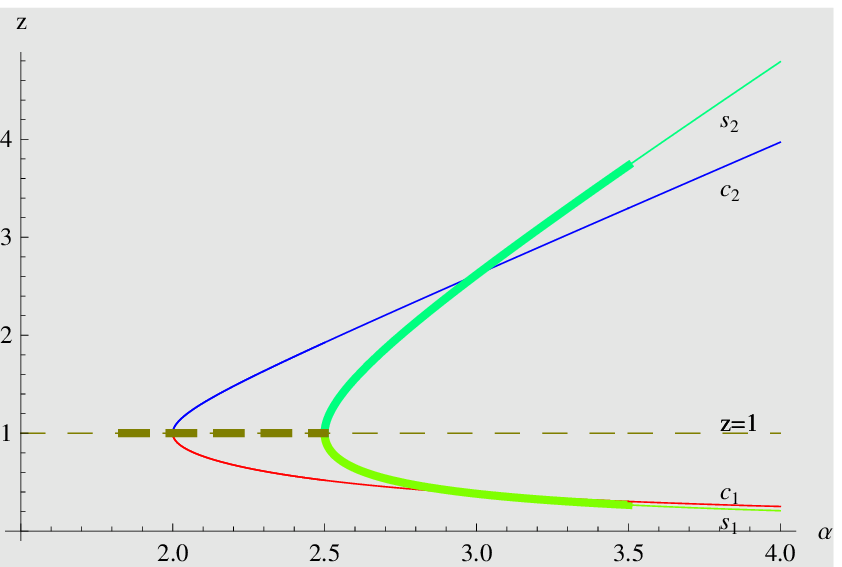}\\
  \caption{Dynamical behavior of strange fixed points and critical points for $1< \alpha <4$}\label{ext}
  \label{estranos y criticos}
  \end{center}
  \end{figure}

Moreover, we can see that when $\alpha \rightarrow 1$, $c_{1}$ tends
to $0$ and $c_{2}$ tends to $\infty $. This fact explains that $
O_{p}\left( z\right) =z^{4}$ when $\alpha=1$ (super-Halley's method), and the only superattractive fixed points were $0$
and $\infty$.

Finally, if $\alpha \rightarrow  \pm \infty $, $c_{1}$ tends to $0$ and $c_{2}$ tends to $\pm \infty $ and $O_{p}\left( z\right) =z^{2}$.

\section{The parameter space}

It is easy to see that the dynamical behavior of operator $O_p(z)$
depends on the values of the parameter $\alpha$. In Figure \ref{planoparametro}, we can see the
parameter space associated to family (\ref{ChH}): each point of the parameter plane is associated to a complex value of $\alpha$, i.e.,
to an element of family (\ref{ChH}). Every value of $\alpha$ belonging to the same connected component of the parameter space give rise to
subsets of schemes of family (\ref{ChH}) with similar dynamical behavior.
\begin{figure}[h]
\begin{center}
  \includegraphics[width=7cm]{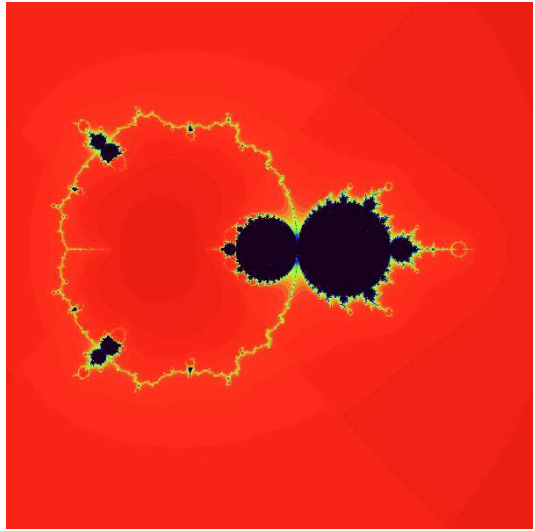}\\
  \caption{Parameter plane}\label{planoparametro}
   \end{center}
\end{figure}

In this parameter space we observe a black figure (let us to call it \emph{
the cat set}), with a certain similarity with the known Mandelbrot set (see \cite{devaney}): for values
of $\alpha $ outside this cat set we will see, numerically, that the Julia set is disconnected. The two
disks in the main body of the cat set correspond to the $\alpha$
values for those the fixed points $z=1$ (the head)  and $z=s_{1}$ and $z=s_{2}$ (the body)
become attractive. Let us observe that the head and the body are surrounded by bulbs, of different sizes, that yield to the appearance of
attractive cycles of different periods.

We also observe a closed curve that passes through the cat's neck,
we call it \emph{the necklace}. As we will
prove in the following, the dynamical planes for values of $\alpha$ inside this curve are
topologically equivalent to disks.

\subsection{The head of the cat set}

The head of the cat corresponds to the values of parameter $\alpha$ for which the
fixed point $z=1$ become attractive, that is, the values of $\alpha$ such that $\left|
\alpha -\frac{13}{6}\right| <\frac{1}{3}.$

In this case, the fixed point $z=1$ is an attractor (Proposition \ref{lemaestabilidad1}) and the other two fixed
points $s_{1}$ and $s_{2}$ are repulsors (Proposition \ref{lemaestabilidadext}). Depending on the values of the
parameter the critical points $c_{1}$ and $c_{2}$ have different behaviors around $z=1$. In Figure \ref{estranos y criticos} we observe
the behavior of the critical points $c_1$ and $c_2$ for real values of parameter $\alpha$ in the interval $(\frac{11}{6},2)$.
For $\frac{11}{6}<\alpha <2$ both critical points are complex. When $\alpha =2$,
both critical points coincide with the fixed point $z=1,$ so that, it is an
superattractor. 



%

\subsubsection{The boundary of the head }

As it has been established in Proposition \ref{lemaestabilidad1}, the boundary of the previous set ($\left| \protect\alpha -\frac{
13}{6}\right| =\frac{1}{3}$) is the loci of bifurcation of the fixed point $z=1$. This fixed point is parabolic on this boundary and
yields to the appearance of attractive cycles, as it happens in Mandelbrot set when we move into the bulbs (see \cite{devaney}).

In this region, $\alpha =\frac{13}{6}+
\frac{1}{3}e^{i\theta }$ and the operator (\ref{op}) can be expressed as:
\begin{equation}
\allowbreak O_{p}\left( z\right) =z^{3}\frac{z-\left( \frac{7}{3}+\frac{2}{3}
e^{i\theta }\right) }{1-\left( \frac{7}{3}+\frac{2}{3}e^{i\theta }\right) z}.
\end{equation}

The first derivative is:
\begin{equation*}
O_{p}^{\prime }\left( z\right) =z^{2}\frac{-63+134z-18e^{i\theta
}-63z^{2}+56ze^{i\theta }-18z^{2}e^{i\theta }+8ze^{2i\theta }}{\left(
-3+7z+2ze^{i\theta }\right) ^{2}}
\end{equation*}
and it is easy to check that $z=1$ is an parabolic point for all these
values of $\alpha$, since
\begin{equation*}
O_{p}^{\prime }\left( 1\right) =\frac{2e^{i\theta }+1}{2+e^{i\theta }} \ \ \mbox{and} \ \    \left| O_{p}^{\prime }\left( 1\right) \right| =1.
\end{equation*}

Therefore, for different values of $\theta$, we find the different
bulbs with attractive cycles surrounding ''the head of the cat''

In this boundary there exist two points that are specially interesting: they correspond to the intersection with the real axe: $\theta =\pi$ and $\theta =0$.
In the first case, $\alpha =\frac{11}{6}$, the operator has the expression
\begin{equation*}
O_{p}\left( z\right) =\allowbreak z^{3}\frac{3z-5}{3-5z}.
\end{equation*}
For this value of $\alpha$, the two strange fixed points $s_{i}$ $i=1,2$ are repulsive. The point $z=1$ is parabolic (since $
O_{p}^{\prime }\left( 1\right) =-1$), and it is in the common boundary of two
parabolic regions 
: the
elements of the orbit corresponding to an initial estimation in one of these parabolic regions, go alternatively from one
region to the other while approaching to the parabolic point $z=1$. 
If $\alpha$ is close to but lower than $\frac{11}{6}$, $z=1$ is a
repulsive fixed point (see Proposition \ref{lemaestabilidad1}). As
it happens in Mandelbrot set, when we take a value of $\alpha$ in
this bulb, an attractive cycle of period 2 appears.

For the case $\theta =0,$ $\alpha =\frac{5}{2}$ and
\begin{equation*}
O_{p}\left( z\right) =\allowbreak z^{3}\frac{z-3}{1-3z}.
\end{equation*}
Now, the three strange fixed points are the same, $z=s_{1}=s_{2}=1,$
and it is parabolic, $\left| O_{p}^{\prime }\left( 1\right) \right|
=1$. We know, by the Flower Theorem of Latou (see \cite{milnor}, for
example), that this parabolic point is in the common boundary of two
attractive regions. The orbits of initial estimations inside each
region approach to $z=1$ without leaving its region.
These attractive areas contain the respective critical points $c_{1}=\frac{2}{9}\left( \frac{11}{2}
-\sqrt{10}\right)$ and $c_{2}=\frac{2}{9}\left(
\frac{11}{2}+\sqrt{10}\right)$.

\subsection{The body of the cat set}

As we have said before, the body of the cat set corresponds to values of the
parameter such that $\left| \alpha -3\right| <\frac{1}{2}.$ In this case,
\begin{equation*}
O_{p}\left( z\right) =z^{3}\frac{z-2\left( \alpha -1\right) }{1-2\left(
\alpha -1\right) z},
\end{equation*}
the fixed point $z=1$ is a repulsor (Proposition \ref{lemaestabilidad1}) and $s_{i}$, $i=1,2$ are attractors
(Proposition \ref{lemaestabilidadext}). So, they have their own basins of attraction with a critical point in each one
(see Figure \ref{estranos y criticos}).

We know that, for $\alpha =3$,
\begin{equation*}
O_p\left( z\right) =\allowbreak z^{3}\frac{z-4}{1-4z},
\end{equation*}
$z=1$ is a repulsor and $s_{i}=c_{i}$, $i=1,2$ are superattractors. 
%

\subsubsection{The boundary of the body}

\smallskip
Similarly to what happen in the Mandelbrot set, the boundary of the
cat set is exactly the bifurcation locus of the family of Chebyshev-Halley
operators acting on quadratic polynomials; that is, the set of parameters
for which the dynamics changes abruptly under small changes of $\alpha$.

This boundary correspond to values  $\left|\alpha
-3\right| =\frac{1}{2}$, that is, $\alpha = 3+\frac{1}{2}e^{i\theta }$.

The strange fixed points $s_i$, $i=1,2$ are parabolic. So, the
different values of the argument $\theta$ give the bifurcation points for the
different bulbs surrounding the body of the cat set.

Two interesting values of $\alpha$ are the intersection between this boundary and the real axe: they correspond to $\theta =0$ and $\theta =\pi$.
If $\theta =0$,  $\alpha =\frac{7}{2}$ and
\begin{equation*}
O_{p}\left( z\right) =\allowbreak z^{3}\frac{z-5}{1-5z}.
\end{equation*}
In this case, $z=1$ is a repulsor and the two strange fixed points $s_{1}$ and $s_{2}$ are parabolic.
Each of these points is in the common boundary of
two parabolic regions, the iterations of the orbit of an initial estimation inside one of these regions go alternatively from one
area to the other, while approaching to the parabolic point. 

In the parameter space, the bulb corresponding to values of $\alpha$ bigger and close to $\frac{7}{2}$ is the loci of the cycles of period 2.

When $\theta =\pi$, $\alpha =\frac{5}{2}$ and this value of the parameter (see Figure \ref{planoparametro})
corresponds to the intersection between the boundaries of the body and the head of the cat. So, the point $\alpha =\frac{5}{2}$
is a bifurcation point and the dynamics changes when the parameter $\alpha $ varies in a small interval around $\frac{5}{2}$ 
and it has been studied in the previous section.

\subsection{Inside the necklace}

For values of the parameter inside the necklace, by applying Propositions \ref{lemaestabilidad1} and \ref{lemaestabilidadext},
the only superattractive fixed points are $0$ and $\infty $; the Julia set is connected  but we see that for different
values of $\alpha $  the fixed points $0$ and $\infty $ have one connected component in each basin of attraction.
As we will see in Proposition \ref{zn}, if $\frac{1}{2}<\alpha<\frac{3}{2}$, these connected components are disks in Riemann sphere.
In other cases, they are topologically equivalent to disks. In the following section, we prove this statement
for $\left| \alpha -1\right| <\frac{1}{2}$, although it can be checked numerically that it is also true in the rest of the area inside the necklace.

\subsubsection{The region $\left| \alpha -1\right| \leq \frac{1}{2}$}\label{discos}

In this area,
\begin{equation*}
O_{p}\left( z\right) =z^{3}\frac{z-2\left( \alpha -1\right) }{1-2\left(
\alpha -1\right) z},
\end{equation*}
where $\alpha =1+re^{i\theta },$ $r\leq \frac{1}{2}$.

\begin{proposition}\label{zn}
If $\left| \alpha -1\right| \leq \frac{1}{2}$, then the dynamical plane is the same as the one of $z^{n}.$
\end{proposition}

\proof If $\alpha \in \mathbf{R}$, the map $m(z)=\frac{z-2\left( \alpha -1\right) }{1-2\left( \alpha -1\right) z}$
is a M\"{o}ebius map that sends the unit disk in the unit disk. There are
only two basins of attraction, of $0$ and  $\infty $. We are going to prove that the critical point
$c_{1}$ is inside the basin of attraction of $0$ and $c_{2}$ is inside the basin of attraction of $\infty$, that is, we need to prove that
\begin{equation*}
\left| c_{1}\right| =\left| \frac{3-4\alpha +2\alpha ^{2}-\sqrt{-6\alpha
+19\alpha ^{2}-16\alpha ^{3}+4\alpha ^{4}}}{3\left( \alpha -1\right) }%
\right| <1,
\end{equation*}
or equivalently,
\[
\left| 3-4\alpha +2\alpha ^{2}-\sqrt{-6\alpha +19\alpha
^{2}-16\alpha ^{3}+4\alpha ^{4}}\right| <3\left| \left( \alpha -1\right)
\right|.
\]

For $ \frac{1}{2}<\alpha <\frac{3}{2}$, it is easy to prove that $ 3-4\alpha +2\alpha ^{2}-\sqrt{-6\alpha +19\alpha
^{2}-16\alpha ^{3}+4\alpha ^{4}}>0$. So, we need to consider only two cases:
\begin{itemize}
\item[i)]  if $\alpha >1$, we want to demonstrate that
$3-4\alpha +2\alpha ^{2}-\sqrt{-6\alpha +19\alpha ^{2}-16\alpha ^{3}+4\alpha ^{4}}<3\left( \alpha -1\right)$
is verified. By simplifying, it is equivalent to
$\left( 2\alpha -3\right) \left( \alpha -2\right)
<\alpha \left( 2\alpha -1\right)$ and this is equivalent to $6<\allowbreak 6\alpha $.
\item[ii)]  if $\alpha <1$, then we need to prove that
$3-4\alpha +2\alpha ^{2}-\sqrt{-6\alpha +19\alpha ^{2}-16\alpha ^{3}+4\alpha ^{4}}<3\left( -\alpha
+1\right) $ and, in an analogous way as before, it is easy to prove that $\allowbreak 6\alpha <6.$
\end{itemize}

Moreover, as $\left| c_{2}\right| =\dfrac{1}{\left| c_{1}\right| }$ then $%
\left| c_{2}\right| >1.$ So, $c_{1}$ is in the basin of $0$ and $c_{2}$ in
the basin of $\infty .$ The Julia set is the unit circle that divides these
two basins.

If $\alpha$ is a complex number, the map $m\left( z\right) =\frac{z-2\left( \alpha -1\right) }{%
1-2\left( \alpha -1\right) z}$ is not a M\"{o}ebius map, but it is holomorphic. Let us analyze the mapping of unit circle by $m$.
The pole of this map is $z^*=\dfrac{1}{2\left( \alpha -1\right) }$ and, in this case, $|z^*|>1$.

Let $z=x+iy$ be a complex number such that $\left| z\right| =1$ and $\alpha =a+ib$, then, $\left( a-1\right) ^{2}+b^{2}\leq \dfrac{1}{2}$.
Let us see the value of $\left|
m\left( z\right) \right| =\left| \frac{z-2\left( \alpha -1\right) }{1-2\left( \alpha -1\right) z}\right|$.

\begin{eqnarray*}
  \left| z-2\left( \alpha -1\right) \right| ^{2}&=&(x-2a+2)^{2}+\left(y-2b\right) ^{2} \\
                                                &=&1+4\left( \left( a-1\right) ^{2}+b^{2}\right) -4ax+4x-4yb
\end{eqnarray*}
\begin{eqnarray*}
 \left| 1-2\left( \alpha -1\right) z\right| ^{2}&=& \left( 1-2ax+2x+2yb\right) ^{2}+\left( -2ay+2y-2bx\right) ^{2} \\
                                                &=&1+4\left( \left( a-1\right)^{2}+b^{2}\right) -4ax+4x+4yb
\end{eqnarray*}

We observe that both are equal if and only if $b=0$, that
is, in the real case. If $yb>0$ then $\left| m\left( z\right) \right| <1$
and $yb<0$ implies $\left| m\left( z\right) \right| >1.$ As $m\left( z\right) $ is
a holomorphic function, the image of the unit circle is a closed curve that
separates the images of the points inside the unit circle from those that
are outside it. So, the dynamical plane for the values of the parameter inside this range consists on two
regions of attraction, $\mathcal{A}\left( 0\right) $ and $\mathcal{A}\left(
\infty \right) $ separated by this closed curve. As before, by continuity
each critical point is in one of these regions.
\endproof

Therefore, it follows that the dynamical plane of this operator is, for the given values of $\alpha$, equivalent to the one of $z^{n}$.
In particular, we observe that, for quadratic polynomials:
\begin{itemize}
\item for $\alpha =1$, $O_{p}\left( z\right) = z^{4}$,
\item if $\alpha =\frac{1}{2}$ then $O_{p}\left( z\right) =z^{3}$
\item and, for $\alpha =\frac{3}{2}$, $O_{p}\left( z\right) =-z^{3}$.
\end{itemize}

Let us remark that dynamical plane associated to every iterative algorithm whose value of $\alpha$ is inside the necklace, is
topologically equivalent to the previous one. 

\subsection{On the necklace}

We focus our attention on the real values of $\alpha$ included in the necklace. In particular, in those which belong to the antennas of the cat set.
If $0\leq \alpha <\dfrac{1}{2}$ and $\alpha \in \mathbb{R}$, we move in the
left antenna of the necklace in the parameter space, that is in the boundary of the cat set. We prove in the following result that,
in the dynamical plane associated to the iterative methods defined by these values of $\alpha$, there are infinite
connected components of the basins of attraction, corresponding to the immediate bassins of attraction and their pre-images.

\begin{proposition}
The dynamical plane for values of $\alpha \in \mathbb{R}$ and $0\leq \alpha< \frac{1}{2}$ consists in two basins of attraction,
$\textit{A}(0)$ and $\textit{A}(\infty)$ with infinity pre-images.
\end{proposition}

\begin{proof}
If $0\leq \alpha <\dfrac{1}{2}$ and $\alpha \in \mathbb{R}$, we move in the
left antenna on the necklace, that is in the boundary of the cat set. For these values, the strange fixed points are repulsive, so they belong
to the Julia set. Moreover, the critical points verify:
\begin{equation}
\left| c_{1}\right| ^{2}=\dfrac{\left( 3-4\alpha +2\alpha ^{2}\right)
^{2}+\left( \sqrt{-\alpha \left( \alpha -2\right) \left( 2\alpha -1\right)
\left( 2\alpha -3\right) }\right) ^{2}}{9\left( \alpha -1\right) ^{2}}= 1
\label{criticosvalen1}
\end{equation}
and, as $c_{2}=\dfrac{1}{c_{1}},$ $\left|c_{2}\right|= 1$. So both critical points are on unit circle.

Moreover, the operator (\ref{op}) has a pole in $z^*=\dfrac{1}{2\left( \alpha -1\right) }$ such that $\left| z^* \right| < 1$;
so, there is an image of $\infty $ inside the unit disk and, by symmetry, there is an image of zero outside the
unit disk. So, by the Theorem of Fatou (see \cite{milnor}), they have infinity basins of
pre-images.
\end{proof}

We observe the same dynamical behavior for values of the parameter in the
right antenna of the cat. The reason is that for $\dfrac{3}{2}< \alpha <2$ (whose values include the ones of the right antenna)
and $\alpha \in \mathbb{R}$  we can use the relationship (\ref{criticosvalen1}).

The case of  $\alpha =0$ has the same operator on quadratic polynomial that the ones studied by the authors in \cite{cordero}
for other different iterative methods. The dynamical plane is similar to the previously described.

\subsection{Outside the cat set}

The cat set, as the Mandelbrot set, could also be defined as the
connectedness locus of the family of  rational functions of Chebyshev-Halley
methods. That is, it is the subset of the complex plane consisting of those
parameters for which the Julia set of the  corresponding dynamical plane is connected. All the
dynamical behaviors we have studied for values of the parameter outside the
cat set show disconnected Julia sets. 
%

\section{Conclusions}

We have studied the dynamics of the Chebyshev-Halley family when it is applied on quadratic polynomials. We
have obtained the fixed and critical points and their dynamical behavior, and we have showed that strange fixed points appear which are attractive
for some values of $\alpha$ (Propositions \ref{lemaestabilidad1} and \ref{lemaestabilidadext}). This means that, for these values, these iterative methods
have basins of attraction different of the roots of the polynomial. So,  the initial point must be chosen carefully.

From the parameter plane obtained, we have observed the cat set, with some similarities with Mandelbrot set:
the head and the body of the cat are surrounded by bulbs. For values of the parameters inside the bulbs, different attractive cycles appear.
We have also studied the dynamical behavior of the family for  values of $\alpha$ inside the necklace and we have shown that it is
analogous to the dynamical behavior of Newton's method (Proposition 3). For  values of $\alpha$ on the antennas of the cat the dynamical
plane has only two basins of attraction, but these basins have infinitely many components (Proposition 4). Finally,
we have obtained numerically that the Julia set is disconnected for values of $\alpha$ outside the cat set.

The cat set is a fascinating creature of complex dynamics. Similarly to the Mandelbrot set, there is a lot of
different dynamics for this cat. We study some of its properties in this
paper, but we are aware that there are plenty of unresolved issues, for
example, it is connected the cat set? We conjecture that the cat set is connected.

We have also observed little cats in the necklace. What's about the dynamics for
these values of the parameter? Even more, where are these cats exactly?.
What happens in the antennas for non real values of $\alpha$?

\bigskip

\noindent {\bf Acknowledgement}

The authors would like to thank Mr. Francisco Chicharro for his valuable
help with the numerical and graphic tools for drawing the dynamical planes.

\end{document}